\documentclass{llncs}
\usepackage[czech,english]{babel}
\usepackage{amsmath,amssymb}
\usepackage{graphicx}
\usepackage{ifthen}
\newboolean{IsProc}\setboolean{IsProc}{true}

\title{Obstacle Numbers of Planar Graphs\thanks{P. Ossona de Mendez  was supported by grant ERCCZ LL-1201 
 and by the European Associated Laboratory ``Structures in
Combinatorics'' (LEA STRUCO).
P. Valtr was supported by project CE-ITI no. P202/12/G061 of the Czech Science Foundation (GA\v CR).
}}
\titlerunning{Obstacle Numbers of Planar Graphs} 
\author{John~Gimbel\inst{1} \and Patrice~Ossona~de~Mendez\inst{2} (0000-0003-0724-3729) \and Pavel Valtr\inst{3} (0000-0002-3102-4166)}
\institute{Department of Mathematics and Statistics, 
University of Alaska Fairbanks, Fairbanks, U.S.A.
  \email{jggimbel@alaska.edu}
  \and 
Centre d'Analyse et de Math\'ematiques Sociales (CNRS, UMR 8557), Paris, France and
     Computer Science Institute of 
     Charles University (IUUK), Prague, Czech Republic,
  \email{pom@ehess.fr}
  \and 
  Department of Mathematics of 
  Charles University (KAM) and CE-ITI, Prague, Czech Republic, 
  \email{valtr@kam.mff.cuni.cz}}

\authorrunning{J. Gimbel, P. Ossona de Mendez, and P. Valtr} 

\newcommand{\pobs}{{\text{\rm pobs}}}
\newcommand{\obs}{{\text{\rm obs}}}
\newcommand{\trfaces}{\text{\rm tr-faces}}
\newcommand{\conv}{{\text{\rm conv}}}

\begin{document}

\maketitle

\begin{abstract}
Given finitely many connected polygonal obstacles $O_1,\dots,O_k$ in the plane and a set $P$ of points in general position and not in any obstacle, the {\em visibility graph} of $P$ with obstacles $O_1,\dots,O_k$ is the (geometric) graph with vertex set $P$, where two vertices are adjacent if the straight line segment joining them intersects no obstacle.

 The {\em obstacle number} of a graph $G$ is the smallest integer $k$  such that $G$ is the visibility graph of a set of points with $k$ obstacles. If $G$ is planar, we define the
 {\em planar obstacle number} of  $G$ by further requiring that the visibility graph has no crossing edges (hence that it is a planar geometric drawing of $G$).
 
 In this paper, we prove that the maximum planar obstacle number of a planar graph of order $n$ is $n-3$, the maximum being attained (in particular) by maximal bipartite planar graphs. 
  
 This displays a significant difference with the standard obstacle number, as we prove that the obstacle number of every  bipartite planar graph (and more generally in the class PURE-2-DIR of intersection graphs of straight line segments in two directions) of order at least $3$ is $1$.
 \end{abstract}

\section{Introduction}
 
 Let $O_1,\dots,O_k$ be closed connected polygonal obstacles in the plane, and let $P$ be a finite set of points not in any obstacle.
 We assume that all the points in $P$ and all vertices of the polygons  $O_1,\dots,O_k$  are in {\em general position}, that is that no three of them are on a line. Note that by considering some $\epsilon$-neighborhood of the obstacles, we see that considering closed or open obstacles makes no difference (thanks to our general position assumption). 
 Also, allowing or forbidding holes in obstacles makes no difference, as holes disappear if we drill a narrow passage from the outside boundary of an obstacle
 to each hole inside the obstacle.
 
The {\em visibility graph} of $P$ with obstacles $O_1,\dots,O_k$ is the geometric graph with vertex set $P$, where 
two vertices $u,v\in P$ are connected by an edge  (geometrically represented by the segment $uv$) if the segment $uv$ does not meet any of the obstacles. Visibility graphs have been extensively studied, see \cite{de2000computational,Ghosh2007,ORourke1997,o1999open,urrutia2000}. 

Closely related to visibility representation, Alpert {\em et al.}  \cite{Alpert2010} introduced the {\em obstacle number} of a graph $G$, which is the minimum number of obstacles in a visibility representation of $G$. For instance, it is known that bipartite graphs can have arbitrarily large obstacle number \cite{pach2011structure}. 
Moreover it was shown in \cite{dujmovic2013obstacle} that there are graphs on $n$ vertices with obstacle number at least $\Omega(n/(\log \log n)^2)$. 
Although there is a trivial quadratic upper bound for the obstacle number, it is an open problem whether the obstacle number has a linear upper bound \cite{Alpert2010,ghosh2013unsolved}. An unexpected, almost linear upper bound ($2n\log n$) was recently obtained in \cite{balko2015drawing}.
For related complexity issues we refer the reader to \cite{chaplick2016obstructing,sarioz2011approximating}.

 However the situation changes when one restricts his attention to planar graphs. Finding a planar graph with obstacle number greater than one was an open problem \cite{Alpert2010,ghosh2013unsolved} until it was recently
 proved \cite{berman2016graphs}  that the icosahedron has obstacle number $2$.
  It is still an interesting open problem to decide whether the obstacle number of planar graphs can be bounded from
above by a constant. This has been verified for outerplanar graphs by Alpert {\em et al.}  \cite{Alpert2010}, who proved that
every outerplanar graph has obstacle number at most $1$ (see also \cite{Fulek2013}).
In this paper we complement this partial result by proving that planar bipartite graphs have obstacle number at most $1$.

When considering a planar graph $G$, the existence of a planar visibility representation of $G$
suggests another definition of the obstacle number. Here
{\em planar visibility representation} means a geometric visibility graph in which no two edges cross (see Fig~\ref{fig:ex1}).
We define a new graph invariant, the {\em planar obstacle number} of $G$, as the minimum number of obstacles in a planar visibility representation of $G$.
(Note that in this context, it will be convenient to assume that obstacles are {\em domains}, that are open connected sets.)
%
%

\begin{figure}
	\begin{center}
		\includegraphics[width=.5\textwidth]{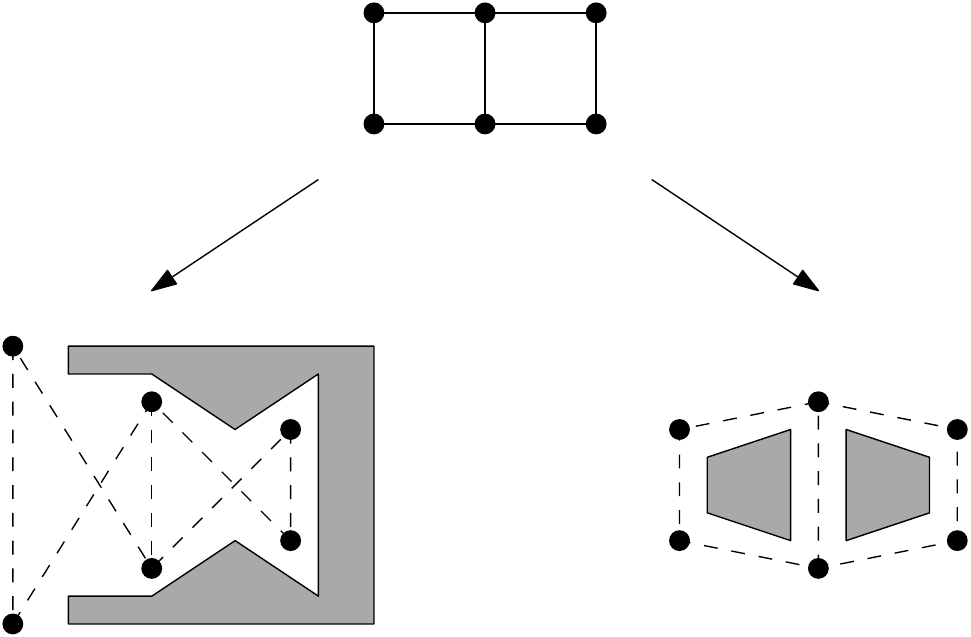}
	\end{center}
	\caption{Example of a graph having a visibility representation with  one obstacle (on the left) and a planar visibility representation with two obstacles (on the right).}
	\label{fig:ex1}
\end{figure}

\section{Our results}
 	For a planar graph $G$, {\em the planar obstacle number of $G$\/},
denoted by $\pobs(G)$, is defined as the minimum number of obstacles
needed to realize $G$ by a straight-line planar drawing of $G$ and a set of obstacles.
 	
 	For $n\in\mathbb{N}$ we further define
 	$$
 	\pobs(n)=\max_{\substack{|G|=n\\G\text{ planar}}}\pobs(G).
 	$$
 
The above definitions are correct due to F\'ary's theorem~\cite{fary}
which says that every planar graph has a straight-line planar drawing; we call
any such drawing a {\em F\'ary drawing\/}.
Slightly perturbing the set of vertices if necessary, we get a F\'ary drawing
with vertices lying in general position. 
Then it is possible to realize the graph by putting a small obstacle
in each face, and this bound can be reduced by one if $G$ is not a tree (in which case, for some embedding of $G$, one can remove the obstacle of the outer face). Hence for a graph $G$ with $n$ vertices and $m$ edges this gives the bound $\pobs(G)\leq \min(m-n+1,1)$.

Obviously, $\pobs(2)=\pobs(3)=1$. In the following theorem we determine $\pobs(n)$ for all $n\ge4$.
 \begin{theorem}\label{t:n-3}
  For $n\ge4$,
    \begin{equation}
    \label{eq:n-3}
 	\pobs(n)=n-3.    	
    \end{equation}
 \end{theorem}

Also we can bound $\pobs(G)$ by means of the number of edges, vertices and triangular faces. For a planar graph $G$, let $\trfaces(G)$ be the maximum number of triangular faces in a planar drawing of $G$.
It is easily checked that one can drop an obstacle in any set of non-adjacent triangles, which improves the trivial upper bound to
\[
\pobs(G)\leq \min(m-n+1-\lceil \trfaces(G)/3\rceil,1).
\] 

\begin{theorem}
	\label{t:mn}
For every connected planar graph $G$  with $n\ge 5$ vertices and $m$ edges, the following inequalities hold:
\begin{equation}
\label{eq:mn}
        \max\{m-n+1-\trfaces(G),1\}\leq \pobs(G)\leq\max\{m-n+1-\lfloor\trfaces(G)/2\rfloor,1\}.
\end{equation}
More precisely, if $\trfaces(G)\leq 1$ we have
\begin{equation}
\label{eq:tf}
\pobs(G)=\begin{cases}
	\max\{m-n+1,1\}&\text{if }\trfaces(G)=0\\
	\max\{m-n,1\}&\text{if }\trfaces(G)=1\\
\end{cases}
\end{equation}
\end{theorem}

In the proof of this theorem, a key argument is that when two adjacent triangles have a non-convex union, the obstacles in these triangles 
may sometimes be dropped,  but not always as witnessed by the following example:
%
	\begin{center}
		\includegraphics[height=.1\textheight]{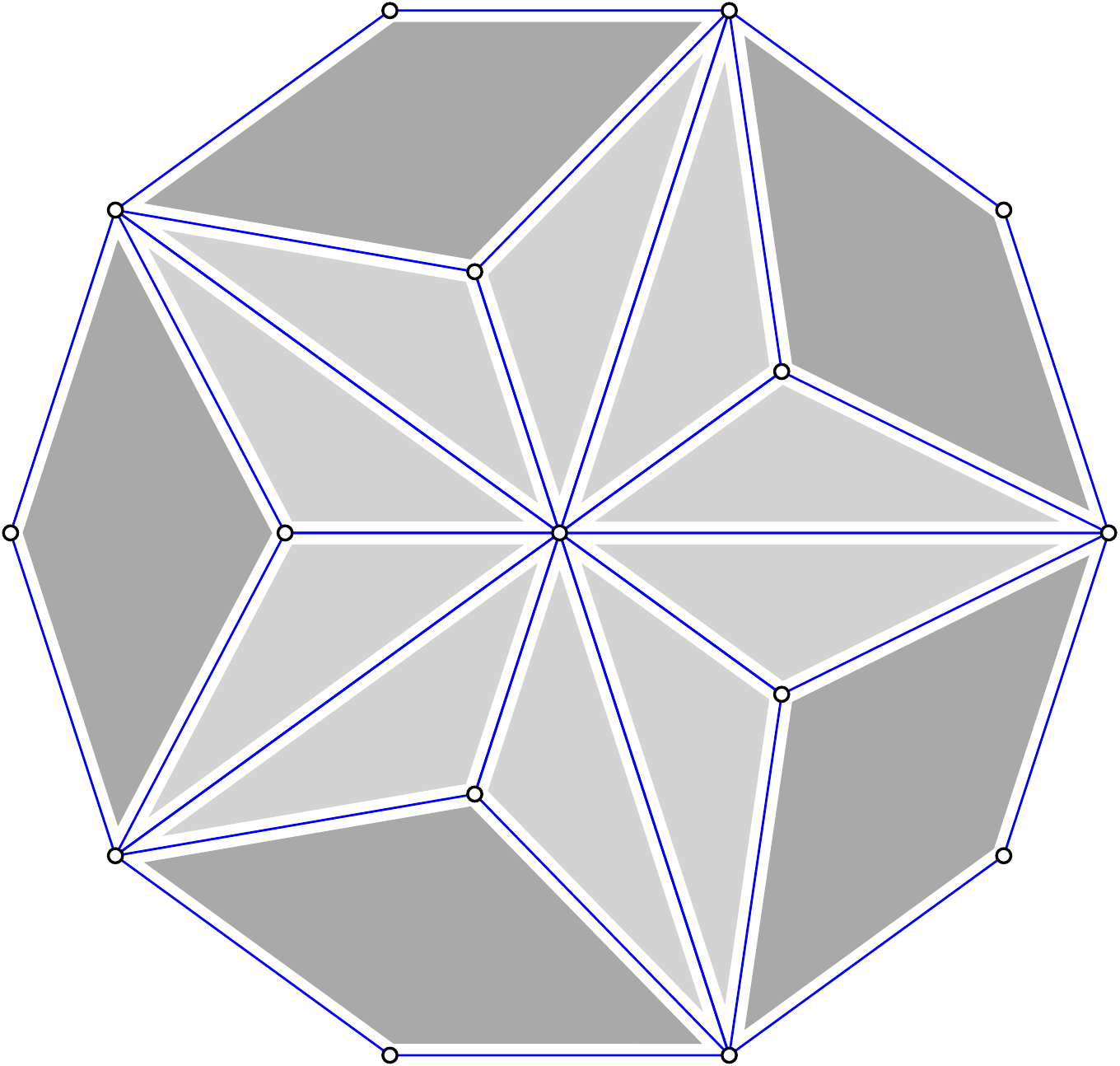}
	\end{center}


Note that a linear time representation algorithm (starting from an embedding of $G$) is easily derived from our proofs. However, choosing an embedding in order to maximize $\trfaces(G)$ is difficult in general, as a reduction from minimum dominating set problem in planar graphs with maximum degree $3$ shows that deciding (for input $(G,k)$) wether $\trfaces(G)\geq m-n-k$ is an NP-hard problem.

%

Planar obstacle number  displays a significant difference with the standard obstacle number, as witnessed by the conjecture that planar graphs have bounded obstacle number.

Using a combinatorial characterization of visibility graphs with a single unbounded obstacle and vertices in convex position (Lemma~\ref{lem:obst1})
  we determine the obstacle  number of PURE-2-DIR graphs, where 
    PURE-2-DIR graphs are intersection graphs of straight line segments with $2$ directions such that any two segments belonging to a same direction are disjoint \cite{km}.

\begin{theorem}
\label{thm:2DIR}
	Let $G$ be a PURE-2-DIR graph (i.e. a graph representable by intersection of straight line segments with $2$ directions with no two segments in a same class intersecting).	 
	Then 
	the obstacle number of $G$ is either $0$ (if $G$ is $K_1$ or $K_2$) or $1$ (otherwise).
\end{theorem}

  As it is known that every bipartite planar graph is a PURE-2-DIR graph, and more precisely has is the contact graph of a family of horizontal and vertical straight line segments in the plane  \cite{Taxi_janos_t} we immediately deduce from Theorem~\ref{thm:2DIR} the value of the planar obstacle number of  bipartite planar graphs:
\begin{theorem}
	\label{thm:bip}
	The obstacle number of a bipartite planar graph $G$ is either $0$ (if $G$ is $K_1$ or $K_2$) or $1$ (otherwise).
\end{theorem}

\section{Preliminaries}
In this paper, we consider graphs that are finite, simple, and loopless.

A {\em drawing} of a graph $G$ in the plane consists in assigning distinct points of the plane to the vertices of $G$, and arcs connecting points to edges. A drawing is {\em simple} if any two arcs intersect at most once, and it is {\em geometrical} if the arcs are straight line segments.

A {\em planar} graph is a graph that  has a drawing in the plane in which no two arcs cross. Such a drawing is called an {\em embedding} of the graph. A planar graph embedded in the plane is called a {\em plane graph}.
A {\em face} of a plane graph is a connected component of the complement of the drawing in the plane. 
A {\em face} is {\em bounded} (or an {\em inner face}) if it has a bounded diameter, it is {\em unbounded} otherwise. Note that every plane graph  has exactly one unbounded face, which is called the {\em outer face}. The traversal of the boundary of a face then produces a sequence of edges (precisely of arcs representing edges), whose length is the {\em length} of the face. Note that if the graph is $2$-connected, the boundary of each face is a cycle, whose length is the length of the face (in general, bridges are counted twice).  Each face of a (loopless simple) $2$-connected graph has thus length at least $3$.
A {\em topological embedding} of a planar graph $G$ in the plane is the equivalence class of an embedding of $G$ in the plane under homeomorphisms of the plane. As they are invariant by every homomorphism of the plane, one can speak about the faces, inner faces, and outer face of a topological embedding of a planar graph in the plane.
It is well known that a topological embedding in the plane is fully determined by the cyclic order of the incident edges around each vertex of the graph and the specification of the outer face. For more informations and background on topological graph embeddings we refer the reader to \cite{moharthom}.

Every (simple) planar graph has  a {\em F\'ary drawing}, that is a crossing free geometric drawing \cite{fary} (see also \cite{steinitz1916polyeder}), and that such a drawing --- with points in a linear size grid --- can be computed in linear time \cite{Taxi_v_pack,Taxi_v_pack2}.  A planar drawing of a graph in which every face is a $3$-cycle (i.e. has length $3$) is a {\em triangulation}.
A {\em maximal planar graph} is a planar graph $G$ to which no edge can be added while preserving the planarity. It is easily checked that a planar graph is maximal if and only if  some/every planar drawing of it is a triangulation.

By definition, every planar visibility representation of a graph $G$ defines a F\'ary drawing of $G$,  each obstacle
lying inside a single face of the drawing. It follows that
any obstacle can be extended to any connected subset of the interior of the face.
 Thus, when minimizing the number of obstacles
we may assume that every face $f$ contains at most one obstacle, and that if $f$ contains an obstacle, this obstacle is large enough to intersect every segment  intersecting $f$ which
connects two independent vertices of $G$.
Note that such a planar visibility representation is thus characterized by a F\'ary drawing of $G$ and the set of its faces containing an obstacle.

Observe that if $F(G)$ is a F\'ary drawing of a connected planar graph $G$
and $H$ is an inner face of $F(G)$ which is not bounded by a $3$-cycle
then some straight-line diagonal of $H$ lies entirely in $H$ and therefore
any realization of $G$ using $F(G)$ must contain an obstacle lying inside the face $H$.

 It follows from Tutte's spring theorem~\cite{tutte} that if a cycle $C$ of a planar graph $G$ bounds some face of some planar drawing of $G$, then $G$ has a F\'ary drawing
such that the boundary of the outer face is the cycle $C$ drawn as the boundary of a convex polygon.

\section{Proof of Theorems~\ref{t:n-3} and \ref{t:mn}}
In this section we derive Theorems~\ref{t:n-3} and \ref{t:mn} from Lemma~\ref{l:key} stated below.
Lemma~\ref{l:key} is then proved in the next section.

We first prove a weaker statement than Theorem~\ref{t:mn}, which determines the planar obstacle number
for connected planar graphs not admitting more than one triangular face in any planar drawing, as given by \eqref{eq:tf}, as well as the lower bound of \eqref{eq:mn}.

 \begin{lemma}\label{l:trianglefree}
For every connected planar graph $G$  with $n\ge 5$ vertices and $m$ edges, it holds
        $$\max\{m-n+1-\trfaces(G),1\}\leq \pobs(G)\leq\max\{m-n+1,1\},$$
where equality with the upper bound holds if and only if $m\leq n$ or $\trfaces(G)=0$.

Thus if $\trfaces(G)\leq 1$ then
        $$\pobs(G)=\max\{m-n+1-\trfaces(G),1\}.$$
 \end{lemma}
\begin{proof}
       For the lower bound, remark that
       every planar visibility representation defines a F\'ary drawing
with $m-n+1$ inner faces. In this drawing, 
every
inner face with length at least $4$ has to contain an obstacle.
Thus if $n\geq 5$ (so that $\pobs(G)>0$) it holds
$\pobs(G)\geq\max\{m-n+1-\trfaces(G),1\}$.

If a planar graph $G$ is acyclic, it obviously holds $\pobs(G)\leq 1$. Otherwise, there exists a planar embedding of $G$ with a face $C$ which is a cycle. Using Tutte's spring embedding, there exists a F\'ary drawing of $G$ whose outer face is a convex polygon. Putting an obstacle in each of the $m-n+1$ inner faces we get a planar visibility representation of $G$. Thus     for every planar graph $G$ on $n\geq5$ vertices it holds
       $$\pobs(G)\leq\max\{m-n+1,1\},$$
       and equality holds if and only if $G$ has no embedding with a triangular face or $m\leq n$. Indeed, assume $m>n$.
Consider an embedding of $G$ with an inner triangular face $t=(v_1,v_2,v_3)$.
At most one more face can contain all of $v_1,v_2,v_3$, for if three faces would contain all of $v_1,v_2,v_3$,  adding a vertex in each of these faces adjacent all of $v_1,v_2,v_3$ would lead to a planar embedding of the non-planar graph $K_{3,3}$.
As $m>n$ the embedding has at least three faces, including one face $f$ which does not contain all of $v_1,v_2,v_3$. We may assume that $f$ is the outer face.
As $G$ is not acyclic, $f$ includes (at least) one cycle $C$.
For each cut-vertex $u$ of $C$, we flip the connected components of $G-u$
that do not contain $C$ into an inner face different from $t$
(such a face exists, as otherwise $f$ would contain all the three vertices of $t$).
After this is done, the outer face is $C$ and $t$ is an inner face.
       Using Tutte's spring embedding one obtains a F\'ary drawing of $G$, in which $C$ is the outer face (drawn as a convex polygon) and $t$ is an inner triangular face. Putting an obstacle in every inner face different from $t$ we get a planar visibility representation of $G$ with $m-n$ obstacles.
\qed\end{proof}
\subsection{Three lemmas}\label{s:lemmas}

Here we state three lemmas used in the proof of both Theorems~\ref{t:n-3} and \ref{t:mn}. We first give several definitions.

 An edge $e$ of a F\'ary drawing of a planar graph $G$ is said to be {\em concave\/} if
 it lies in two triangular faces and the union of these two faces is a non-convex quadrilateral.
 (If $G$ is a triangulation then such an edge is sometimes called a {\em non-flippable edge\/} in the literature.)
 
 Suppose $T$ is a triangulation, $x,y,z$ are three vertices of $T$, and $X$ is a subset of the edge set of $T$.
Then we say that a F\'ary drawing of $T$ is {\em $(X;x,y,z)$-concave\/} if $x,y,z$ are the vertices of the outer face
and all the edges of $X$ are concave.

For a graph $G=(V,E)$, an edge set $E'\subseteq E$ is said to be {\em sparse in $G$\/} if each cycle in $G$ contains
at least two edges of $E\setminus E'$.


\begin{lemma}\label{l:2coloring}
For every triangulation $T$ there is a $2$-coloring of the faces of $T$ such that the set
of edges contained in a pair of faces of the same color is sparse in $T$. 
\end{lemma}

\begin{proof}
For a $2$-coloring $\chi$ of the faces of $T$, let $S(\chi)$ be the set of edges contained in a pair
of faces of the same color. Let $\overline{\chi}$ be a $2$-coloring of the faces of $T$ minimizing the size of 
$S(\overline{\chi})$, and let $F(T)$ be any fixed (topological) planar drawing of $T$. Any cycle $C$ of $T$
contains at most $\lfloor |E(C)|/2\rfloor\le |E(C)|-2$ edges of $S(\overline{\chi})$,
since otherwise the size of $S(\overline{\chi})$ could be decreased by flipping the colors
of all the faces lying inside the cycle $C$ in the drawing $F(T)$.
Thus, $S(\overline{\chi})$ is sparse and $\overline{\chi}$ satisfies Lemma~\ref{l:2coloring}.
\qed\end{proof}

\begin{lemma}\label{l:key}
 Let $T$ be a triangulation having a face with vertices $a$, $b$, $c$, and let $S\subseteq E(T)$ be a sparse set of edges of $T$.
 Then $T$ has an $(S;a,b,c)$-concave drawing.\qed
\end{lemma}

The proof of Lemma~\ref{l:key} is sketched in Section~\ref{s:lemmakey}; all details will be included in the full version.
Importance of $(S;a,b,c)$-concave drawings clearly appears in the following technical lemma.
 \begin{lemma}
\label{l:rep}
Let $T$ be a planar triangulation, let 
$G$ be a spanning subgraph of $T$ which includes all the edges of a face bounded by a triangle $\{a,b,c\}$ of $T$,
let $\chi$ be a $2$-coloring of the faces of $T$ such that the set 
 $S$  of edges of $T$ contained in a pair of faces of the same color is sparse,  let $P(T)$ be  an $(S;a,b,c)$-concave drawing of $T$, and 
 let $P(G)$ be the F\'ary drawing of $G$ obtained from $P(T)$ by the removal of all the edges of $E(T)\setminus E(G)$.

Then by placing an obstacle in each face of $P(G)$ that is not a triangular inner face with color $2$ we get a planar visibility representation of $G$.
\end{lemma}

\begin{proof}
	Let $x,y$ be a pair of non-adjacent vertices of $G$, and let $F_1,\dots,F_t$ be the sequence
of the faces of $P(G)$ intersected in this order by the segment $xy$
(faces may appear more than once in this sequence), and let
$\overline{F_1},\dots,\overline{F_t}$ be their closure.
 We have to show that at least one of the faces
$F_1,\dots,F_t$ is not a triangular face colored $2$.
Suppose for contradiction this is not the case and that $F_1,\dots,F_t$ are triangular faces colored $2$.
\begin{itemize}
	\item If $t=1$ then $F_1$ is not a triangular face, contradicting the assumption.
	\item   if $t=2$,
since $\overline{F_1}\cap \overline{F_2}$ lies in $S$ it is concave in $ P(T)$ (and also in $P(G)$),  and $\conv(\overline{F_2}\cup\{x\})=\conv(\overline{F_1}\cup \overline{F_2})$ is a triangle.
This contradicts the assumption that the segment $xy$ is included in $\overline{F_1}\cup \overline{F_2}$, which necessarily implies that $\overline{F_1}\cup \overline{F_2}$ is a convex quadrilateral.
\item If $t\ge3$  then all the edges
$\overline{F_1}\cap \overline{F_2}, \overline{F_2}\cup \overline{F_3}$ lie in $S$ and therefore the boundary of $F_2$ contains two edges in $S$, contradicting the assumption that $S$ is sparse.
\end{itemize}
\qed\end{proof}

\subsection{The proof}\label{s:proof}

Here we derive Theorems~\ref{t:n-3} and \ref{t:mn} from Lemmas~\ref{l:2coloring},\ref{l:key}, and \ref{l:rep}. Note that according to Lemma~\ref{l:trianglefree} we need to prove only the upper bound of \eqref{eq:mn} to establish Theorem~\ref{t:mn}.
  
 Let $n\ge 4$. By Euler's formula every F\'ary drawing of $K_{2,n-2}$ has $n-3$ inner faces.
Each of them is a quadrilateral and therefore has to contain an obstacle.
 Thus,
 $
 \pobs(n)\ge\pobs(K_{2,n-2})\ge n-3$.
 In order to prove Theorems~\ref{t:n-3} and  \ref{t:mn}
  it remains to prove that for a graph $G$ on $n$ vertices and $m$ edges the following inequality holds: 
  $\pobs(n)\le \min(n-3,m-n+1-\lfloor\trfaces(G)/2\rfloor)$.

 First we show that it suffices to prove the upper bound 
 for all connected planar graphs on $n\ge4$ vertices. This will easily follows from the next lemma
 which relies on an invariant closely related to $\pobs(G)$.
 The invariant $\pobs'(G)$ is defined as the minimum number of obstacles lying in inner faces of
 a realization of $G$, where the minimum is taken over all possible realizations of $G$.
 It follows from the definition that $\pobs'(G)=\pobs(G)$ if every realization of $G$ with $\pobs(G)$ obstacles
 uses only obstacles in the inner faces. Otherwise we have $\pobs'(G)=\pobs(G)-1$.
 
 \begin{lemma}\label{l:disconn}
Let $G$ be a disconnected planar graph, and let $A_1,\dots,A_\alpha$ be the components of $G$.
$$
\pobs(G) = \begin{cases}
\sum_{i=1}^\alpha\pobs'(A_i)&\text{if   $\exists A_j$ with $\pobs(A_j)=\pobs'(A_j)>0,$}\\
	1 + \sum_{i=1}^\alpha\pobs'(A_i)&\text{otherwise.}
\end{cases}
$$
%
%
\end{lemma}

\begin{proof}
For a realization of $G$ with the given number of obstacles,
suppose first that $\pobs(A_j)=\pobs'(A_j)>0$ for some $A_j$.
We realize the component $A_j$ with  $\pobs(A_j)=\pobs'(A_j)$ obstacles. Inside one of its obstacles
we make small individual holes for the remaining components.
Each component $A_i$, $i\neq j$, is realized in ``its'' hole using $\pobs'(A_i)$ obstacles
placed inside its inner faces. The boundary of the hole surrounds the drawing
of $A_i$ so that the surrounding obstacle functions for $A_i$ as an obstacle in the outside face of $A_i$.
If there is no component with $\pobs(A_j)=\pobs'(A_j)>0$ then we proceed similarly as above,
realizing the components inside individual holes of one big obstacle.

It remains to show that $G$ cannot be realized with a smaller number of obstacles.
Consider a realization of $G$, and let $A_i$ be one of its components.
If we consider only the edges of $A_i$, one or more obstacles must appear in at least
$\pobs'(A_i)$ inner faces of the drawing of $A_i$. Let $F$ be one of these faces.
The interior of $F$ may contain other components of $G$ but (at least) one of the obstacles inside $F$
does not lie in any face of a component of $G$ lying inside $F$, as otherwise a vertex of $F$ would see a vertex
of a component lying inside $F$. Assigning this obstacle to $A_i$, we get at least
$\pobs'(A_i)$ obstacles assigned to $A_i$ and to no other component of $G$.
This gives the lower bound $\pobs(G) \geq \sum_{i=1}^\alpha\pobs'(A_i)$.
Suppose now that there is no $A_j$ with $\pobs(A_j)=\pobs'(A_j)>0$. Then there must be a component $A_i$
incident to the outer face of the drawing of $G$. We have $\pobs(A_i)=\pobs'(A_i)=0$ or
$\pobs(A_i)>\pobs'(A_i)$. It follows that either we need an additional obstacle in the outer face of $A_i$ which does not lie inside
another component of $G$, or at least $\pobs(A_i)\geq\pobs'(A_i)+1$ faces of $A_i$ contain obstacles and we may assign
at least $\pobs'(A_i)+1$ obstacles to the component $A_i$.
\qed\end{proof}

Observe that $\pobs'(C_4)=1$ and $\pobs'(G)=0$ for any other connected graph $G$ on at most four vertices.
Therefore, due to Lemma~\ref{l:disconn} and to the trivial inequality $\pobs'(G)\leq\pobs(G)$,
it suffices to prove the upper bounds in Theorems~\ref{t:n-3} and  \ref{t:mn}
 for connected planar graphs (on $n\ge4$ and $n\ge 5$ vertices, respectively).

  Let $G$ be a connected planar graph on $n\ge4$ vertices.
  We want to show that $\pobs(G)\le n-3$. Let $F(G)$ be a F\'ary drawing of $G$.
  We distinguish the following two cases:
 either $F(G)$ has a triangular face,
  	or no face of $F(G)$ is triangular.

In the latter case, $m\leq 2n-4$ thus $m-n+1\leq n-3$. Hence by Lemma~\ref{l:trianglefree} $\pobs(G)\le n-3$.

In the rest of the proof we suppose that $F(G)$ has at least a triangular face $\{a,b,c\}$.
By adding straight-line edges to $F(G)$ we obtain a F\'ary drawing of a triangulation further denoted $T$.

Let $\chi$ be a $2$-coloring of the faces of $T$ satisfying Lemma~\ref{l:2coloring},
and let $S$ be the set of edges of $T$ contained in a pair of faces of the same color.
Due to Lemma~\ref{l:key}, $T$ has an $(S;a,b,c)$-concave drawing $P(T)$.
Let $P(G)$ be the F\'ary drawing of $G$ obtained from $P(T)$ by the removal of all the edges of $E(T)\setminus E(G)$.

We denote the colors of $\chi$ by $1$ and $2$.
We partition the set of the inner faces of $P(G)$ into the following three subsets:

$I_1:=$ the set of the triangular inner faces of $P(G)$ having color $1$,

$I_2:=$ the set of the triangular inner faces of $P(G)$ having color $2$,

$I_3:=$ the set of the non-triangular inner faces of $P(G)$.

Since $P(T)$ has $2n-5$ inner faces and each face in $I_3$ is the union of at least two faces of $P(T)$, we have
$ |I_1|+|I_2|+2|I_3|\le 2n-5$.
It follows that
$\min\{|I_1|+|I_3|,|I_2|+|I_3|\}\le \lfloor (2n-5)/2\rfloor = n-3$.
Also we have
$|I_1|+|I_2|=\trfaces(G)-1$
and
$|I_3|=m-n+2-\trfaces(G) $. 
Hence $|I_1|+|I_2|+2|I_3|=2(m-n+1)-(\trfaces(G)-1) $, and
$$\min\{|I_1|+|I_3|,|I_2|+|I_3|\}\le m-n+1-\lfloor\trfaces(G)/2\rfloor.$$

Without loss of generality we suppose that
$$|I_1|+|I_3|\le \min(n-3,m-n+1-\lfloor\trfaces(G)/2\rfloor).$$
To prove $\pobs(G)\le \min(n-3,m-n+1-\lfloor\trfaces(G)/2\rfloor)$ --- from which Theorems~\ref{t:n-3} and \ref{t:mn}
follow ---
 it now suffices to prove that taking the F\'ary drawing $P(G)$
and placing an obstacle in each face of $I_1\cup I_3$ gives a representation of $G$, what follows from Lemma~\ref{l:rep}.
This finishes the proof of Theorem~\ref{t:n-3}. 
%
The proof of Lemma~\ref{l:key}, which will be included in the full version of the manuscript, is only sketched here.

\subsection{Sketch of the proof of Lemma~\ref{l:key}}\label{s:lemmakey}
We proceed by induction on $n$. The case $n=3$ is trivial and the case $n=4$ is easy. Suppose now that $n>4$.
If $S$ is empty or contains only an edge of the triangle $abc$
then any F\'ary drawing of $T$ with the outer face $abc$ is $(S;a,b,c)$-concave.
Suppose now that an edge $e\in S$ is not an edge of the triangle $abc$.
Let $\text{Tr}$ be the set of all triangles of $T$ containing the edge $e$. We fix a F\'ary
drawing $F(T)$ of $T$ with the outer face $abc$.
The set $\text{Tr}$ can be partitioned into two sets $\text{Tr}^+$ and $\text{Tr}^-$, such that $\text{Tr}^+$
contains the triangles of $\text{Tr}$ lying on one side of $e$ and $\text{Tr}^-$ contains the triangles of $\text{Tr}$
lying on the other side of $e$. Since $T$ is a triangulation and $e$ is an inner edge in $F(T)$, the set
$\text{Tr}^+$ contains a unique face $t^+$ of $F(T)$. Similarly, $\text{Tr}^-$ contains
a unique face $t^-$ of $F(T)$. Let $t^{++}$ be the unique triangle in $\text{Tr}^+$ containing all
the other triangles of $\text{Tr}^+$ in the drawing $F(T)$. The triangle $t^{--}$ is defined analogously.
Note that if $|\text{Tr}^+|=1$ then $t^{++}=t^+$. Analogously,
if $|\text{Tr}^-|=1$ then $t^{--}=t^-$.

\begin{center}
	\includegraphics[height=.15\textheight]{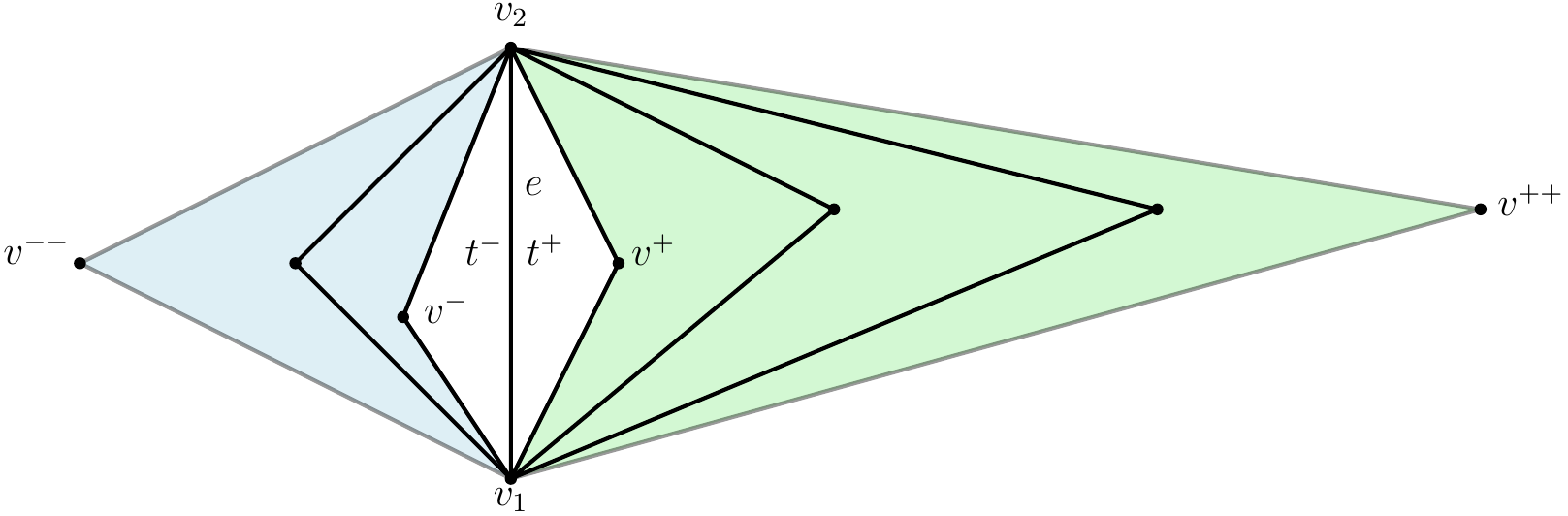}
\end{center}

Without loss of generality, we may assume that in clockwise order one finds $v^{++},v_1,v_2$ for the triangle $t^{++}$, $v^+,v_1,v_2$ for the triangle $t^+$,
 $v^{--},v_2,v_1$ for the triangle $t^{--}$, and  $v^-,v_2,v_1$ for the triangle $t^-$.

We define a triangulation $T\odot e$ as the triangulation obtained from $T$ by {\em the planar
contraction of $e$ with respect to $F(T)$\/}, i.e., it is the graph obtained from $T$ by the following two operations:
(i) removal of all the vertices and edges lying inside the triangles
$t^{++}$ and $t^{--}$ in $F(T)$, and (ii) contraction of $e$.
Thus, $e=v_1v_2$ is contracted to a new vertex $v_e$ and the triangles $t^{++}$ and $t^{--}$
(including their interiors in $F(T)$) are replaced by the new edges $v^{++}v_e$ and $v^{--}v_e$, respectively.
The drawing $F(T)$ and the construction of $T\odot e$ immediately give
a (topological) planar drawing $Dr(T\odot e)$ of $T\odot e$ with each face being a triangle and with the outer face $abc$.
Let $T^+$ denote the triangulation consisting of the vertices and edges
lying in the triangle $t^{++}$ in the drawing $F(T)$.
Analogously, let $T^-$ denote the triangulation consisting of the vertices and edges
lying in the triangle $t^{--}$ in the drawing $F(T)$.
Further, let $S^+:=S\cap E(T^+)$ and $S^-:=S\cap E(T^-)$.

We can prove the following two observations:
(1) $S^+$ is sparse in $T^+$, and $S^-$ is sparse in $T^-$; (2) $S_0$ is sparse in $T\odot e$.

Hence we can apply the inductive hypothesis (i) on $T\odot e$ and $S_0$,
(ii) on $T^+$ and $S^+$, and (iii) on $T^-$ and $S^-$.
In the first case we get an $(S_0;a,b,c)$-concave drawing $P(T\odot e)$ of $T\odot e$ with the outer
face $abc$.
In the second case we get an $(S^+;v^{++},v_1,v_2)$-concave drawing $P(T^+)$ of $T^+$ with the outer
face $t^{++}=v^{++}v_1v_2$.
In the third case we get an $(S^-;v^{--},v_2,v_1)$-concave drawing $P(T^-)$ of $T^-$ with the outer
face $t^{--}=v^{--}v_2v_1$.
Our construction of an $(S;a,b,c)$-concave drawing of $T$ is obtained by properly combining the drawings
$P(T\odot e)$, $P(T^+)$ and $P(T^-)$.
\def\eps{\varepsilon}
Starting with $P(T\odot e)$, we replace the vertex $v_e$ by a short segment $e$ and
the two edges $v_ev^{++}$ and $v_ev^{--}$ by skinny copies  of $P(T^+)$ and $P(T^-)$, respectively,
in such a way that the edge $e$ is concave in the resulting drawing, thus proving  Lemma~\ref{l:key}
 (details omitted here). \hfill\qed

\section{Obstacle number of Intersection Graphs of Segments}

Generally, it is an interesting question to characterize graphs with obstacle number $1$, and specifically those graphs that can be represented using a single unbounded obstacle.

It is easily seen that a graph has such a representation if and only if it is an induced subgraph of the visibility graph of a simple polygon. However, no characterization is known for visibility graphs of  simple polygons. However one can give some interesting characterization in a special case:
\begin{lemma}
\label{lem:obst1}
A graph $G$ has a representation as a visibility graph with a single unbounded obstacle and vertices in convex position if and only if there exist functions $p,I$ mapping the vertex set of $G$ to points  (resp. to circular arcs)
of $\mathbb S^1$, in such a way that
\begin{itemize}
	\item for every vertex $v$ of $G$ it holds $p(v)\notin I(v)$;
	\item for every distinct vertices $u,v$ of $G$, it holds that $u$ and $v$ are adjacent if and only if  $p(u)\in I(v)$ and $p(v)\in I(u)$.
\end{itemize}
\end{lemma}
\begin{proof}
	The existence of a visibility representation from functions $p$ and $I$ is clear, and follows the same ideas as the representation for intersection graphs of horizontal and vertical segments.
	
	Conversely, assume that $G$ is representable as a visibility graph using a single unbounded obstacle, and consider such a representation. Without loss of generality, we can assume that the obstacle is the exterior of a simple polygon, on which lie the vertices of $G$.
	 Let $v_1,\dots,v_n$ be the vertices of $G$ in the order in which they appear on the polygon (say clockwise), and let $i\neq j$. 
	 Let us now prove that $v_i$ and $v_j$ are adjacent if and only if there exist $i_1,i_2$ and $j_1,j_2$ such that
	 \begin{itemize}
	 	\item in circular order, $i$ is between $i_1$ and $i_2$ 
	 	and $j$ is between $j_1$ and $j_2$;
	 	\item $v_i$ is adjacent to both $v_{j_1}$ and $v_{j_2}$
	 	and $v_j$ is adjacent to both $v_{i_1}$ and $v_{i_2}$.
	 \end{itemize} 
	 If $v_i$ and $v_j$ are adjacent, one can let $i_1=i_2=i$ and $j_1=j_2=j$. Conversely, if $i_1,i_2,j_1,j_2$ have the properties above, then let $x$ (resp. $y$) be the intersection of segments $[v_i,v_{j_2}]$ and $[v_j,v_{i_1}]$ (resp. of segments $[v_i,v_{j_1}]$ and $[v_j,v_{i_2}]$). Then the region $R$ delimited by the quadrangle $(x,v_i,y,v_j)$ does not meet the obstacle, thus $v_j$ is visible from $v_i$, that is $v_i$ and $v_j$ are adjacent.
	 
	 Let us now define functions $p$ and $I$: for $1\leq k\leq n$ let 
	 $p(v_k)=e^{ik2\pi/n}$ and $I(v_k)$ be the inclusion minimum circular arc containing all the neighbours of $v_k$ but not $v_k$. It follows from the property above that the functions $p$ and $I$ satisfy the requirements of the Lemma.
	 \qed\end{proof}

\begin{figure}[h]
	\begin{center}
		\includegraphics[width=.4\textwidth]{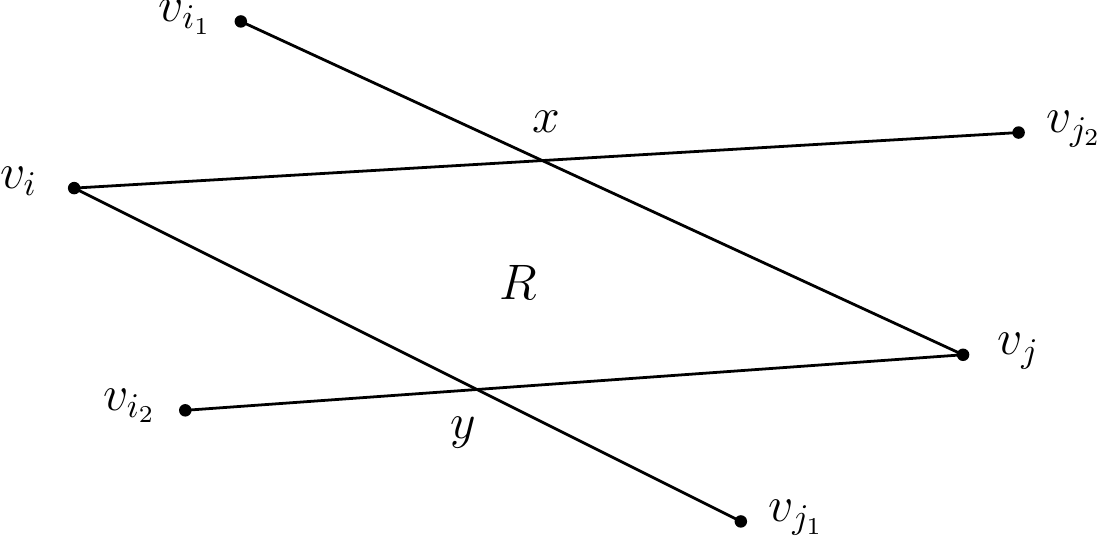}
	\end{center}
	\caption{Illustration for the proof of Lemma~\ref{lem:obst1}.}
	\label{fig:obst1}
\end{figure}

We now prove Theorem~\ref{thm:2DIR}, which states that
	the obstacle number $\obs(G)$ of a  PURE-2-DIR graph $G$ can be computed as follows:
	\[
	\obs(G)=\begin{cases}
	0&\text{if $G$ is $K_1$ or $K_2$},\\
	1&\text{ otherwise.}
	\end{cases}
	\]
	and then deduce Theorem~\ref{thm:bip}.	
	
\begin{proof}[Proof of Theorem~\ref{thm:2DIR}]
	Let $G$ be a PURE-2-DIR graph. Without loss of generality, we can assume that segments are either horizontal or vertical, and that they are numbered from  top to bottom and left to right. Hence, denoting $a_1,\dots,a_p$ the vertices corresponding to the horizontal segments and by 
	$b_1,\dots,b_q$ the vertices corresponding to the vertical segments, it is easily checked that $a_i$ is adjacent to $b_j$ if and only if there exist $i_1\leq i\leq i_2$ and $j_1\leq j\leq j_2$ such that $a_{i_1}$ and $a_{i_2}$ are adjacent to $b_j$ and $b_{j_1}$ and $b_{j_2}$ are adjacent to $a_i$. The result then follows from Lemma~\ref{lem:obst1} (see
	Fig.~\ref{fig:2dir} for an example of a single obstacle representation of a PURE-2-DIR graph).
	\qed\end{proof}

\begin{figure}
	\begin{center}
		\includegraphics[width=\textwidth]{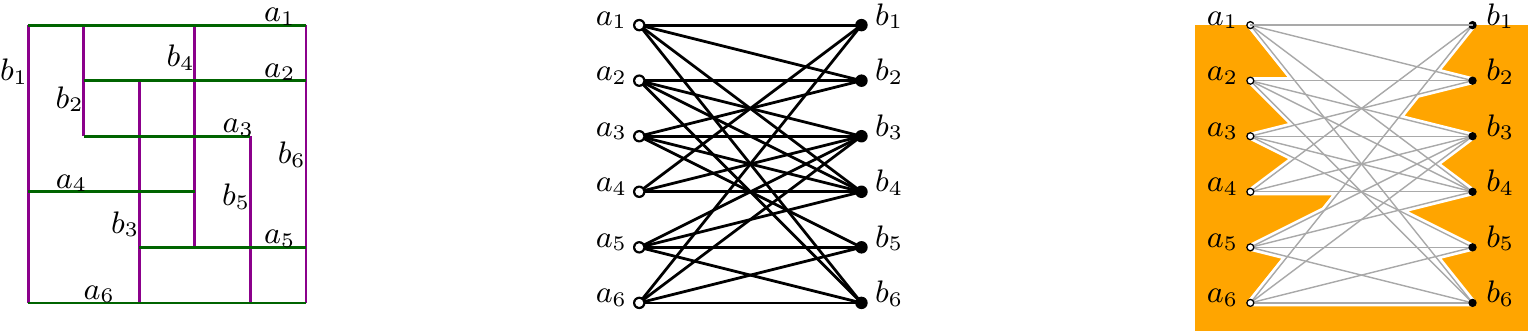}
	\end{center}
\caption{Construction of a representation of a graph $G\in\text{$2$-DIR}$ with a single obstacle.}
\label{fig:2dir}
\end{figure}

\begin{proof}[Proof of Theorem~\ref{thm:bip}]
 	It has been proved in \cite{Taxi_janos_t} that every bipartite planar graph can be represented as PURE-2-DIR graph. Thus it directly follows from Theorem~\ref{thm:2DIR} that the obstacle number of every planar bipartite graph different from $K_1$ and $K_2$ is $1$.
\qed\end{proof}
Note that the linear time algorithm to compute a representation of a bipartite planar graph as a PURE-2-DIR graph given in \cite{Taxi_left} can be easily modified to provide a linear time algorithm providing a representation of a planar bipartite graph as the visibility graph of a family of points with a single polygonal obstacle.

\subsubsection{Acknowledgments}
 We thank V{\'\i}t Jel{\'\i}nek for ideas leading to a simplification of our proof.
We also thank Daniel Kr\'a{\v l}  
and Roman Nedela for inspiring comments on our research.



\end{document}